\documentclass[12pt]{amsart}

\author{Paul \textsc{Poncet}}
\address{CMAP, \'{E}cole Polytechnique, Route de Saclay, 91128 Palaiseau Cedex, France} 

\email{poncet@cmap.polytechnique.fr}

\usepackage[colorinlistoftodos,bordercolor=orange,backgroundcolor=orange!20,linecolor=orange,textsize=scriptsize]{todonotes}
\usepackage{stmaryrd}
\usepackage{amssymb}
\usepackage{amsmath}
\usepackage{graphicx}
\usepackage{t1enc}
\usepackage{color}
\usepackage[latin1]{inputenc}

\usepackage{calrsfs} 

\usepackage{times} 
\usepackage{bbm}
\usepackage{ifpdf}

\newtheorem{theorem}{Theorem}[section]
\newtheorem{corollary}[theorem]{Corollary}
\newtheorem{proposition}[theorem]{Proposition}

\theoremstyle{definition}

\newtheorem{example}[theorem]{Example}

\newtheorem{remark}[theorem]{Remark}

\def\cprime{$'$}

\begin{document}

\title{A memo on bornologies and size functions}

\date{\today}



\begin{abstract}
We recall the notion of abstract bornology, and connect it with topological spaces and size functions. 
As a generalization of measures of non-compactness, we show how every size function can be mapped to a maxitive measure.  
\end{abstract}

\maketitle

\section{Bornologies}

Given a set $E$, a \textit{bornology} on $E$ is a collection $\mathrsfs{B}$ of subsets of $E$ with the following properties: 
\begin{itemize}
	\item $\mathrsfs{B}$ covers $E$; 
	\item if $B, B' \in \mathrsfs{B}$ then $B \cup B' \in \mathrsfs{B}$; 
	\item if $B \subseteq B'$ and $B' \in \mathrsfs{B}$ then $B \in \mathrsfs{B}$. 
\end{itemize}
The elements of $\mathrsfs{B}$ are called the \textit{bounded} subsets. 
The pair $(E, \mathrsfs{B})$ is a \textit{bornological set}. 
Note that all finite subsets of $E$ are necessarily bounded. 

A map between bornological sets is \textit{bounded} if the direct image of every bounded subset is bounded. 
The bornological sets together with bounded maps as morphisms form a category. 

\begin{example}
The collection of finite subsets of a set is a bornology. 
More generally, on a topological space, a subset is \textit{relatively compact} if it is included in a compact subset; 
then the collection of relatively compact subsets is a bornology. 
We call it the \textit{Heine--Borel} bornology and denote it by $\mathrsfs{H}$. 
\end{example}

\begin{example}
Let $(X, d)$ be a metric space. 
Then the collection of subsets included in an open ball is a bornology. 
\end{example}

\begin{example}
Let $(P, \leqslant)$ be a partially ordered set. 
Assume that $P$ is \textit{directed}, i.e.\ that, for all $x, y \in P$, there exists some $z \in P$ with $x \leqslant z$ and $y \leqslant z$. 
An \textit{upper bound} of a subset $A$ is an element $x \in P$ such that $a \leqslant x$ for all $a \in A$. 
Then the collection of subsets with an upper bound is a bornology. 
\end{example}


Given a map $f : E \to E'$ and a bornology $\mathrsfs{B}'$ on $E'$, the following collection defines a bornology on $E$: 
\[
f^{-1}(\mathrsfs{B}') = \{ B \subseteq E : f(B) \in \mathrsfs{B}' \}. 
\]
Moreover, ${ \mathrsfs{H} } \subseteq { f^{-1}(\mathrsfs{H}') }$ always holds if $E$ and $E'$ are topological spaces and $f$ is a continuous map. 


A topological space equipped with a bornology is \textit{locally bounded} if, for every $x \in E$ and every open subset $G$ with $G \ni x$, there exists a bounded neighborhood $B$ of $x$ such that $B \subseteq G$. 

\begin{proposition}
Let $E$ be a topological space, and $\mathrsfs{B}$ be a bornology on $E$. 
If $E$ is locally bounded, then every (relatively) compact subset of $E$ is bounded, i.e.\ ${ \mathrsfs{H} } \subseteq { \mathrsfs{B} }$. 
\end{proposition}

\begin{proof}
Let $K$ be a compact subset. 
If $x \in K$, there is an open bounded subset $B_x$ such that $x \in B_x$ since $E$ is locally bounded. 
So $K$ is covered by the open family $(B_x)_{x \in K}$, and since $K$ is compact there exists a finite set $F$ of $K$ such that ${ K } \subseteq { \bigcup_{x \in F} B_x }$. The latter subset is bounded as a finite union of bounded subsets, hence $K$ is bounded. 
\end{proof}

\section{Size functions}

We write $\mathbb{R}$ (resp.\ $\mathbb{R}_+$) for the set of real numbers (resp.\ nonnegative real numbers), and $\overline{\mathbb{R}}_+$ is a shorthand for $\mathbb{R}_+ \cup \{ \infty \}$. 
A \textit{size function} on a set $E$ is a map $\rho : 2^E \to \overline{\mathbb{R}}_+$ such that 
\begin{itemize}
  \item $\rho(\emptyset) = 0$;
  \item $\rho(\{ x \}) < \infty$, for all $x \in E$;
  \item $\rho(A) \leqslant \rho(A')$, for all subsets $A \subseteq A'$ of $E$.
\end{itemize}
The extended real number $\rho(A)$ is called the \textit{size} of $A$. 
We say that $\rho$ is \textit{total} if $\rho(\{ x \}) = 0$, for all $x \in E$. 

\begin{example}\label{ex:alphabeta}
Let $E$ be a metric space. 
Then maps $\alpha$ and $\beta$ defined by 
\begin{align*}
\alpha(A) &= \inf \{ t > 0 : A \mbox{ is included in some ball of radius $t$ } \}, \\
\beta(A) &= \inf \{ t > 0 : A \mbox{ has a diameter less than $t$ } \},
\end{align*}
for all $A \subseteq E$, are total size functions on $E$. 
\end{example}

The \textit{monk} of a size function $\rho$ on $E$ is the map $\rho^{\circ}$ defined on $2^E$ by 
\begin{align*}
\rho^{\circ}(A) = \inf\{ t > 0: A \mbox{ is covered by finitely many subsets of size } \leqslant t \}, 
\end{align*}
for all subsets $A$ of $E$. 
A subset $A$ of $E$ is \textit{$\rho$-bounded} if $\rho^{\circ}(A) < \infty$, and \textit{totally $\rho$-bounded} if $\rho^{\circ}(A) = 0$, or equivalently if, for every $s > 0$, there are finitely many subsets of $E$ covering $A$ with size less than $s$. 

\begin{proposition}\label{prop:rhoprops}
Let $\rho$ be a size function on a set $E$. 
Then the following properties hold: 
\begin{enumerate}
	\item\label{prop:rhoprops1} $\rho^{\circ}(A) \leqslant \rho(A)$, for all subsets $A$ of $E$; 
	\item\label{prop:rhoprops2} $\rho^{\circ}(\{ x \}) = \rho(\{ x \})$, for all $x \in E$;
	\item\label{prop:rhoprops3} $\rho^{\circ}(A \cup A') = \max(\rho^{\circ}(A), \rho^{\circ}(A'))$, for all subsets $A$, $A'$ of $E$; 
	\item\label{prop:rhoprops4} $(\rho^{\circ})^{\circ}(A) = \rho^{\circ}(A)$, for all subsets $A$ of $E$.
\end{enumerate}
In particular, $\rho^{\circ}$ is a size function, total if $\rho$ is total.   
\end{proposition}

\begin{proof}
\eqref{prop:rhoprops1} is obvious. 
Note that $\rho^{\circ}(A) \leqslant \rho^{\circ}(A')$ whenever $A \subseteq A'$. 
Combined with \eqref{prop:rhoprops1}, this implies that $\rho^{\circ}$ is a size function. 

\eqref{prop:rhoprops2}. 
Let $x \in E$. 
Let $t > \rho^{\circ}(\{ x \})$. 
Then $\{ x \}$ can be covered by a finite collection of subsets of size less than $t$. 
So $x \in A$ for some subset of size less than $t$. 
Thus, $\rho(\{ x \}) \leqslant t$. 
This shows that $\rho(\{ x \}) \leqslant \rho^{\circ}(\{ x \})$. 
The reverse inequality holds by \eqref{prop:rhoprops1}, so that $\rho(\{ x \}) = \rho^{\circ}(\{ x \})$, as required. 

\eqref{prop:rhoprops3}. 
Let $t > \max(\rho^{\circ}(A), \rho^{\circ}(A'))$. 
Then both $A$ and $A'$ can be covered by a finite collection of subsets of size less than $t$. 
Thus, so does $A \cup A'$. 
This proves that $t \geqslant \rho^{\circ}(A \cup A')$. 
This yields $\rho^{\circ}(A \cup A') \leqslant \max(\rho^{\circ}(A), \rho^{\circ}(A'))$. 
The reverse inequality is clear since $\rho^{\circ}$, so we get the desired result. 

\eqref{prop:rhoprops4}. 
Let $A \subseteq E$. 
We already know that $(\rho^{\circ})^{\circ}(A) \leqslant \rho^{\circ}(A)$ by \eqref{prop:rhoprops1}. 
Now let $t > (\rho^{\circ})^{\circ}(A)$. 
Then there are some subsets $A_1, \ldots, A_n$ of $E$ such that ${ A } \subseteq { \bigcup_{i = 1}^n A_i }$ and $\rho^{\circ}(A_i) \leqslant t$ for all $i$. 
Since $\rho^{\circ}$ is maxitive we get $\rho^{\circ}(A) \leqslant \rho^{\circ}(\bigcup_{i = 1}^n A_i) \leqslant t$. 
This proves that $\rho^{\circ}(A) \leqslant (\rho^{\circ})^{\circ}(A)$. 
\end{proof}

Since $\rho^{\circ}$ is a size function satisfying Proposition~\ref{prop:rhoprops}\eqref{prop:rhoprops3}, this is a \textit{maxitive measure} (see e.g.\ Poncet \cite{Poncet17} for a survey). 
It happens that this is the greatest maxitive measure less than $\rho$, as the following result testifies. 

\begin{corollary}
Let $\rho$ be a size function $\rho$ on a set $E$. 
Then the map $\rho^{\circ}$ is the greatest maxitive measure on $2^E$ below $\rho$. 
\end{corollary}

\begin{proof}
Let $\nu$ be a maxitive measure on $2^E$ such that $\nu(A) \leqslant \rho(A)$, for all $A \subseteq E$. 
Let $A \subseteq E$ and $t > \rho^{\circ}(A)$. 
Then there are some subsets $A_1, \ldots, A_n$ of $E$ such that ${ A } \subseteq { \bigcup_{i = 1}^n A_i }$ and $\rho(A_i) \leqslant t$ for all $i$. 
This yields $\nu(A_i) \leqslant t$ for all $i$, and since $\nu$ is maxitive we get $\nu(A) \leqslant \nu(\bigcup_{i = 1}^n A_i) \leqslant t$. 
Thus, $\nu(A) \leqslant \rho^{\circ}(A)$, as required. 
\end{proof}

\begin{corollary}
Let $\rho$ be a (total) size function on a set $E$. 
Then the collection of (totally) $\rho$-bounded subsets is a bornology on $E$. 
\end{corollary}

Given a size function $\rho$ on a topological set $E$, we say that $\rho$ is \textit{closed} if $\rho(\overline{A}) = \rho(A)$, for all subsets $A$ of $E$, where $A \mapsto \overline{A}$ denotes the topological closure operator. 

\begin{proposition}
Let $\rho$ be a size function on a topological space $E$. 
If $\rho$ is closed, then $\rho^{\circ}$ is closed. 
\end{proposition}

\begin{proof}
Let $t > \rho^{\circ}(A)$. 
Then ${ A } \subseteq { \bigcup_{i = 1}^n A_i }$, for some subsets $A_i$ of size less than $t$. 
This yields ${ \overline{A} } \subseteq { \bigcup_{i = 1}^n \overline{A}_i }$, and $\rho(\overline{A}_i) = \rho(A_i) \leqslant t$. 
Hence, $t \geqslant \rho(\overline{A})$. 
This proves that $\rho^{\circ}(\overline{A}) = \rho^{\circ}(A)$, for all subsets $A$ of $E$. 
\end{proof}

Given a size function $\rho$ on a topological set $E$, we write $\rho^{+}$ for the map defined on $2^E$ by
\[
\rho^{+}(A) = \inf_{G \supseteq A} \rho(G),
\]
for all $A \subseteq E$, where $G$ runs over the open subsets of $E$ containing $A$. 
We say that $\rho$ is 
\textit{weakly outer-continuous} if $\rho$ and $\rho^{+}$ agree on compact subsets. 

\begin{proposition}\label{prop:compact}
Let $\rho$ be a size function on a topological space $E$. 
Then 
\[
\rho^{\circ}(K) \leqslant \sup_{x \in K} \rho^{+}(\{ x \}),
\]
for all compact subsets $K$ of $E$. 
If moreover $\rho$ is weakly outer-continuous, then 
\[
{ \rho^{\circ}(K) } = { \sup_{x \in K} \rho(\{ x \}) } < { \infty },
\]
for all compact subsets $K$ of $E$. 
\end{proposition}

\begin{proof}
Let $K$ be a compact subset of $E$, and let $t > \sup_{x \in K} \rho^{+}(\{ x \})$. 
If $x \in K$, we have $t > 0 = \rho^{+}(\{x\}) = \bigwedge_{G \ni x} \rho(G)$, where the infimum is taken over the open subsets $G$ containing $x$. 
So there is some open subset $G_x \ni x$ such that $t > \rho(G_x)$. 
Now $K$ is compact and covered by the open family $(G_x)_{x\in K}$, so there exists a finite subset $F$ of $K$ such that $K \subseteq { \bigcup_{x \in F} G_x }$. 
This proves that $\rho^{\circ}(K) \leqslant t$. 
Thus, $\rho^{\circ}(K) \leqslant \sup_{x \in K} \rho^{+}(\{ x \})$. 

If moreover $\rho$ is weakly outer-continuous, then $\rho^{+}(\{ x \}) = \rho(\{ x \}) = \rho^{\circ}(\{ x \}) \leqslant \rho^{\circ}(K)$, for all $x \in K$, so that 
$\rho^{\circ}(K) \leqslant \sup_{x \in K} \rho^{+}(\{ x \}) = \sup_{x \in K} \rho(\{ x \}) \leqslant \rho^{\circ}(K)$. 
Now, $\rho^{+}(\{ x \}) = \rho(\{ x \}) < \infty$ for all $x$ implies that the map $x \mapsto \rho(\{ x \})$ is upper semicontinuous, hence its sup on $K$ is reached. 
Thus, $\rho^{\circ}(K) < \infty$. 
\end{proof}


\begin{theorem}\label{thm:oc}
Let $\rho$ be a weakly outer-continuous size function on a topological space $E$. 
Then the following properties hold:
\begin{enumerate}
  \item\label{thm:oc1} $E$ is locally bounded with respect to the bornology of $\rho$-bounded subsets.
  \item\label{thm:oc2} If $\rho$ is total, then the relatively compact subsets of $E$ are totally $\rho$-bounded. 
  \item\label{thm:oc3} If $\rho$ is total and $E$ is locally compact, then $E$ is locally bounded with respect to the bornology of totally $\rho$-bounded subsets. 
\end{enumerate}
\end{theorem}

\begin{proof}
\eqref{thm:oc1}. 
Let $x \in E$. 
Let $G$ be an open subset containing $x$. 
Since $\rho(\{ x \}) < \infty$, we can take $t \in \mathbb{R}_+$ such that $t > \rho(\{ x \})$. 
Since $\rho$ is weakly outer-continuous, we have $t > \rho^{+}(\{ x \})$, so there is some open subset $G_x \ni x$ such that $t > \rho(G_x)$. 
This shows that $x \in { U_x } \subseteq { G }$, where $U_x$ is the open subset $G \cap G_x$. 
Moreover, $\rho^{\circ}(U_x) \leqslant \rho(U_x) \leqslant \rho(G_x) < t < \infty$, hence $U_x$ is $\rho$-bounded. 
This shows that $E$ is locally bounded with respect to the bornology of $\rho$-bounded subsets. 

\eqref{thm:oc2} is a direct consequence of Proposition~\ref{prop:compact}. 

\eqref{thm:oc3} is straightforward from the proof of \eqref{thm:oc1} combined with \eqref{thm:oc2}. 
\end{proof}

\begin{remark}
If $E = \mathbb{R}^n$ and $\rho$ is total as in Theorem~\ref{thm:oc}\eqref{thm:oc2}, then $\rho(B) = 0$ for every $B \subseteq E$ bounded in the usual sense. 
\end{remark}

\begin{example}[Example~\ref{ex:alphabeta} continued]
Let $E$ be a metric space. 
While the list of axioms defining \textit{measures of non-compactness} may vary from one author to the other (see e.g.\ Appell \cite{Appell05} or Mallet-Paret and Nussbaum \cite{Mallet-Paret10a, Mallet-Paret10b}), here is a possible minimalistic definition. 
We call \textit{measure of non-compactness} on $E$ a closed maxitive measure $\nu : 2^E \to \overline{\mathbb{R}}_+$ on $E$ such that $\nu(K) = 0$ for all compact subsets $K$ of $E$. 
Classical examples are the \textit{Hausdorff measure} $\alpha^{\circ}$ induced by the size function $\alpha$, and the \textit{Kuratowski measure} $\beta^{\circ}$ induced by the size function $\beta$. 
Note the relations $\alpha^{\circ} \leqslant \beta^{\circ} \leqslant 2 \alpha^{\circ}$. 
\end{example}

\bibliographystyle{plain}

\begin{thebibliography}{1}

\bibitem{Appell05}
J\"urgen Appell.
\newblock Lipschitz constants and measures of noncompactness of some
  pathological maps arising in nonlinear fixed point and eigenvalue theory.
\newblock In {\em Proceedings of the Conference on Function Spaces,
  Differential Operators and Nonlinear Analysis held at Praha, May 27 - June 1,
  2004}, pages 19--27. Math. Inst. Acad. Sci. of Czech Republic, 2005.
\newblock Edited by P. Dr\'abek, J. R\'akosnik.

\bibitem{Mallet-Paret10a}
John Mallet-Paret and Roger~D. Nussbaum.
\newblock Inequivalent measures of noncompactness.
\newblock {\em Annali di Matematica Pura ed Applicata}, 48, 2010.

\bibitem{Mallet-Paret10b}
John Mallet-Paret and Roger~D. Nussbaum.
\newblock Inequivalent measures of noncompactness and the radius of the
  essential spectrum.
\newblock {\em Proc. Amer. Math. Soc.}, 139(3):917--930, 2011.

\bibitem{Poncet17}
Paul Poncet.
\newblock Representation of maxitive measures: {A}n overview.
\newblock {\em Math. Slovaca}, 67(1):121--150, 2017.

\end{thebibliography}

\def\cprime{$'$} \def\cprime{$'$} \def\cprime{$'$} \def\cprime{$'$}
  \def\ocirc#1{\ifmmode\setbox0=\hbox{$#1$}\dimen0=\ht0 \advance\dimen0
  by1pt\rlap{\hbox to\wd0{\hss\raise\dimen0
  \hbox{\hskip.2em$\scriptscriptstyle\circ$}\hss}}#1\else {\accent"17 #1}\fi}
  \def\ocirc#1{\ifmmode\setbox0=\hbox{$#1$}\dimen0=\ht0 \advance\dimen0
  by1pt\rlap{\hbox to\wd0{\hss\raise\dimen0
  \hbox{\hskip.2em$\scriptscriptstyle\circ$}\hss}}#1\else {\accent"17 #1}\fi}

\end{document}